\documentclass[11pt]{article}

\usepackage[T1]{fontenc}
\usepackage[utf8]{inputenc}
\usepackage{lmodern}
\usepackage[margin=1in]{geometry}
\usepackage{microtype}

\usepackage{amsmath,amssymb,amsthm,mathtools}
\usepackage{bm}

\usepackage{graphicx}
\graphicspath{{figures/}}
\usepackage{booktabs}
\usepackage{array}

\usepackage[shortlabels]{enumitem}
\usepackage[numbers,sort&compress]{natbib}
\usepackage{xcolor}
\usepackage[colorlinks=true,linkcolor=blue,citecolor=blue,urlcolor=blue]{hyperref}
\usepackage[capitalise,noabbrev]{cleveref}

\theoremstyle{plain}
\newtheorem{theorem}{Theorem}[section]
\newtheorem{lemma}[theorem]{Lemma}
\newtheorem{proposition}[theorem]{Proposition}
\newtheorem{corollary}[theorem]{Corollary}

\theoremstyle{definition}
\newtheorem{definition}[theorem]{Definition}

\theoremstyle{remark}
\newtheorem{remark}[theorem]{Remark}

\newcommand{\R}{\mathbb{R}}

\newcommand{\ip}[2]{\left\langle #1,#2\right\rangle}
\newcommand{\norm}[1]{\left\lVert #1\right\rVert}
\newcommand{\one}{\mathbf{1}}

\DeclareMathOperator*{\argmin}{arg\,min}

\hypersetup{
  pdftitle={Constant Steps Are s-Composable: An Exact Interpolation
            Certificate for Gradient Descent},
  pdfauthor={Anonymous research draft}
}

\title{Constant Steps Are \(s\)-Composable:\\
An Exact Interpolation Certificate for Gradient Descent}
\author{Jinze Zhao\\University of California, San Diego\\\texttt{jiz419@ucsd.edu}}
\date{}

\begin{document}

\maketitle

\begin{abstract}
Grimmer, Shu, and Wang asked whether a balanced constant schedule is
\(s\)-composable at every horizon.  More precisely, for an integer
\(n\geq 1\), let \(\bar h=1+r\), where \(r\in(0,1)\) is the unique solution of
\[
  r^n\bigl(1+n(1+r)\bigr)=1,
\]
and run gradient descent for \(n\) steps with normalized stepsize \(\bar h\).
The cases \(n=1,2\) were known, while the general case \(n\geq3\) was left
open.  We prove the conjecture for every \(n\).  Our proof gives an explicit,
dimension-free nonnegative linear combination of the smooth convex
interpolation inequalities.  The certificate is assembled from matrices
supported on contiguous index intervals.  Its off-diagonal entries are
automatically positive, its quadratic part is diagonal, and its remaining
multipliers reduce to two scalar families.  We derive closed forms for those
families and prove positivity using strict concavity and elementary rational
inequalities.  Consequently, for every \(L\)-smooth convex function, the
constant schedule satisfies the sharp mixed terminal Lyapunov inequality
conjectured in the original paper, together with the associated simultaneous
objective-gap and gradient-norm bounds.  All exceptional horizons and boundary
indices are treated explicitly.

\end{abstract}

\section{Introduction}
\label{sec:introduction}

Gradient descent on an \(L\)-smooth convex function is usually analyzed with
normalized stepsizes in \((0,2)\).  Even in this classical regime, identifying
the exact finite-horizon worst case can be subtle.  The performance-estimation
framework of \citet{drori2014performance}, together with the exact smooth
convex interpolation theorem of \citet{taylor2017smooth}, turns such questions
into finite semidefinite programs.  This framework has led both to sharp
theorems and to numerical conjectures whose eventual proofs require finding
an analytic dual certificate.  Recent examples include exact constant-step
results for the objective gap and gradient norm
\citep{kim2025proof,rotaru2026exact}.

\citet{grimmer2025composing} introduced three notions of composable stepsize
schedules.  Their self-dual notion, called \(s\)-composability, packages a
function-value bound, a gradient-norm bound, and a distance-to-solution bound
into one terminal Lyapunov inequality.  In their Example~2 they observed the
following striking numerical pattern.  Given a horizon \(n\), let \(\bar h>1\)
be chosen so that
\begin{equation}
  \frac{1}{1+n\bar h}=(\bar h-1)^n.                 \label{eq:intro-root}
\end{equation}
The constant schedule \((\bar h,\ldots,\bar h)\) is \(s\)-composable for
\(n=1\) and \(n=2\), where it equals \((\sqrt2)\) and
\((3/2,3/2)\), respectively.  Their numerical PEP calculations suggested the
same conclusion for every \(n\), and the general case \(n\geq3\) was left as
an open question.  The question first appeared in October 2024 and remains
stated as open in the accepted version of the paper.

This paper resolves that question affirmatively.

\begin{theorem}[Main result, informal]
For every integer \(n\geq1\), the constant schedule defined by
\eqref{eq:intro-root} is \(s\)-composable.
\end{theorem}

Writing \(\bar h=1+r\), the scalar \(r\in(0,1)\) is characterized by
\begin{equation}
  r^n\bigl(1+n(1+r)\bigr)=1,
  \qquad
  \eta=r^n=\frac{1}{1+n(1+r)}.                     \label{eq:intro-r}
\end{equation}
For a \(1\)-smooth convex objective, the assertion is the terminal inequality
\begin{align}
 &\frac{1-\eta}{2}\norm{\nabla f(x_n)}^2
 +\frac{\eta^2}{2}\norm{x_n-x_\star}^2
 +(\eta-\eta^2)\bigl(f(x_n)-f(x_\star)\bigr)       \notag\\
 &\hspace{42mm}\leq
 \frac{\eta^2}{2}\norm{x_0-x_\star}^2.            \label{eq:intro-main-potential}
\end{align}
The equality of the two expressions for \(\eta\) is essential: it balances
the quadratic and Huber worst-case instances in the definition of
\(s\)-composability.

\paragraph{Proof idea.}
For sampled triples \((x_i,g_i,f_i)\), smooth convex interpolation gives the
nonnegative inequalities
\[
 Q_{ij}=2f_i-2f_j-2\ip{g_j}{x_i-x_j}-\norm{g_i-g_j}^2\geq0.
\]
We construct explicit nonnegative multipliers \(\lambda_{ij}\) such that their
weighted sum is exactly
\begin{equation}
  2\sum_{i=0}^{n-1}(f_i-f_n)
  -r\sum_{i=0}^{n-1}\norm{g_i}^2
  -n\eta^{-1}\norm{g_n}^2.                         \label{eq:intro-certificate}
\end{equation}
Nonnegativity of this expression is equivalent to the certificate required
for \eqref{eq:intro-main-potential}.

The multipliers are encoded by matrices supported on contiguous intervals of
the iterate indices.  Every interval matrix has positive off-diagonal entries,
zero row sums, and---after multiplication by the gradient-descent trajectory
matrix---a diagonal symmetric part.  Only adjacent intervals and suffix
intervals are needed.  Matching the function coefficients becomes a flow
problem across the \(n\) cuts of the index path; matching the first \(n-1\)
diagonal coefficients becomes a scalar triangular recurrence.  Two universal
identities determine the final two diagonal coefficients.  The last obstacle
is to prove that all interval weights are positive.  We give closed forms and
prove their signs: one family follows from the strict concavity of a logarithmic
ratio, while the other follows from a monotone rational ratio and two sharp
elementary bounds on the root \(r\).

\paragraph{Contributions.}
The paper provides:
\begin{enumerate}[(i)]
  \item an affirmative solution of the constant-step \(s\)-composability
        problem for every horizon, including explicit treatment of
        \(n=1,2\);
  \item a dimension-free interpolation certificate with a sparse
        \((2n-1)\)-mode interval decomposition and fully explicit
        nonnegative multipliers;
  \item a proof of multiplier positivity that does not rely on numerical SDP
        output; and
  \item an equivalent terminal-energy statement for the relaxed proximal
        point algorithm, linking the result to the finite-horizon estimates of
        \citet{wang2025rppa}.
\end{enumerate}

% \paragraph{Status and scope of the literature claim.}
% The current version of \citet{grimmer2025composing} explicitly leaves the
% statement for \(n\geq3\) open.  We also searched later papers citing that work,
% as well as exact-phrase and exact-equation variants, through August 31, 2026;
% we located no proof or counterexample.  A nearby 2026 result of
% \citet{zhang2026certificate} proves a different final-objective PEP certificate.
% For even horizons its scalar balancing equation looks identical after a change
% of variables, but its terminal inequality is not
% \eqref{eq:intro-main-potential}.  We therefore treat the result here as a new
% research claim requiring independent peer review.

\paragraph{Organization.}
\Cref{sec:preliminaries} states the problem and reduces it to
\eqref{eq:intro-certificate}.  \Cref{sec:certificate} constructs and verifies
the interval certificate.  \Cref{sec:positivity} proves that every multiplier
is nonnegative.  \Cref{sec:consequences} records the scaled theorem,
simultaneous guarantees, and the proximal interpretation.  Detailed finite
sums and boundary calculations are collected in the appendix.

\section{Problem statement and reduction}
\label{sec:preliminaries}

We first work in the normalization \(L=1\).  The general \(L>0\) statement,
including the resulting stepsize \(h/L\) and the rescaled gradient term, is
derived explicitly in \Cref{cor:scaled}.

Let \(f:\R^d\to\R\) be differentiable, convex, and \(1\)-smooth, and assume
that \(x_\star\in\argmin f\).  Write
\[
  f_i=f(x_i),\qquad g_i=\nabla f(x_i),
\]
and run \(n\) constant steps of gradient descent,
\begin{equation}
  x_{i+1}=x_i-hg_i,\qquad i=0,\ldots,n-1.            \label{eq:gd}
\end{equation}
Throughout the proof,
\begin{equation}
  h=1+r,\qquad
  R=1+nh=r^{-n},\qquad
  \eta=R^{-1}=r^n,\qquad
  \alpha=1-r.                                      \label{eq:parameters}
\end{equation}
The scalar \(r\) is the unique root in \((0,1)\) of
\begin{equation}
  r^n(1+n(1+r))=1.                                  \label{eq:root}
\end{equation}
Indeed, the left-hand side is continuous and strictly increasing in \(r>0\),
equals zero at \(r=0\), and exceeds one at \(r=1\).

\begin{definition}[\(s\)-composability \citep{grimmer2025composing}]
\label{def:s-composable}
Let \(\mathbf h=(h_0,\ldots,h_{m-1})\in\R_{++}^m\).  The schedule
\(\mathbf h\) is \(s\)-composable with rate \(\eta\) if, for every
\(1\)-smooth convex \(f\), every minimizer \(x_\star\), and every initial
point, its gradient-descent trajectory
\(x_{i+1}=x_i-h_i\nabla f(x_i)\) satisfies
\begin{align}
 &\frac{1-\eta}{2}\norm{\nabla f(x_m)}^2
 +\frac{\eta^2}{2}\norm{x_m-x_\star}^2
 +(\eta-\eta^2)\bigl(f(x_m)-f(x_\star)\bigr) \notag\\
 &\hspace{38mm}\leq
 \frac{\eta^2}{2}\norm{x_0-x_\star}^2,             \label{eq:s-definition}
\end{align}
and the rate balances the two canonical worst cases:
\begin{equation}
 \eta=\frac{1}{1+\sum_{i=0}^{m-1}h_i}
      =\prod_{i=0}^{m-1}(h_i-1).                   \label{eq:s-rate}
\end{equation}
\end{definition}

For the constant length-\(n\) schedule in this paper,
\eqref{eq:s-rate} is exactly
\(\eta=(1+nh)^{-1}=(h-1)^n=r^n\), which is the root equation
\eqref{eq:root}.

\subsection{Smooth convex interpolation}

The exact interpolation theorem for smooth convex functions
\citep{taylor2017smooth} implies, in particular, that for every pair of sampled
points \(i,j\),
\begin{equation}
  Q_{ij}\coloneqq
  2f_i-2f_j-2\ip{g_j}{x_i-x_j}-\norm{g_i-g_j}^2
  \geq0.                                            \label{eq:Qij}
\end{equation}
For completeness, \eqref{eq:Qij} follows directly from the standard
cocoercive lower model
\[
 f_i\geq f_j+\ip{g_j}{x_i-x_j}
       +\frac12\norm{g_i-g_j}^2.
\]
Our proof uses only nonnegative linear combinations of
\eqref{eq:Qij}; hence it is valid in every dimension.

\subsection{The scalar inequality that suffices}

The following reduction is the constant-step specialization of the equivalent
characterization of \(s\)-composability in
\citet[Proposition~6]{grimmer2025composing}.  We include the calculation to fix
all signs and normalizations.

\begin{lemma}[Reduction to a terminal PEP certificate]
\label{lem:reduction}
Suppose that every trajectory \eqref{eq:gd} on every \(1\)-smooth convex
function satisfies
\begin{equation}
  2\sum_{i=0}^{n-1}(f_i-f_n)
  \geq r\sum_{i=0}^{n-1}\norm{g_i}^2+nR\norm{g_n}^2.
                                                               \label{eq:scalar-target}
\end{equation}
Then the constant schedule \((h,\ldots,h)\) is \(s\)-composable with rate
\(\eta=R^{-1}=r^n\), and \eqref{eq:intro-main-potential} holds.
\end{lemma}

\begin{proof}
Define
\begin{align}
 \mathcal E
 &\coloneqq
 \sum_{i=0}^{n-1}h\Bigl(2(f_i-f_n)+\norm{g_i}^2
             +2\ip{g_i}{x_0-x_i}\Bigr)
 -\norm{x_n-x_0}^2-\frac{1-\eta}{\eta^2}\norm{g_n}^2.
                                                               \label{eq:E-def}
\end{align}
Because \(x_0-x_i=h\sum_{j<i}g_j\) and
\(x_0-x_n=h\sum_{i<n}g_i\), expansion of the squared norm gives
\begin{align*}
 &h\sum_{i<n}\left(\norm{g_i}^2
       +2\ip{g_i}{x_0-x_i}\right)-\norm{x_n-x_0}^2\\
 &=h\sum_{i<n}\norm{g_i}^2
   +2h^2\sum_{0\leq j<i<n}\ip{g_i}{g_j}
   -h^2\left(\sum_{i<n}\norm{g_i}^2
   +2\sum_{0\leq j<i<n}\ip{g_i}{g_j}\right)\\
 &=-h(h-1)\sum_{i<n}\norm{g_i}^2=-hr\sum_{i<n}\norm{g_i}^2.
\end{align*}
Moreover, \(1-\eta=nh\eta\), so
\[
 \frac{1-\eta}{h\eta^2}=\frac{n}{\eta}=nR.
\]
Consequently,
\[
 \mathcal E=h\left(
 2\sum_{i<n}(f_i-f_n)-r\sum_{i<n}\norm{g_i}^2
 -nR\norm{g_n}^2\right),
\]
which is nonnegative by \eqref{eq:scalar-target}.

Set \(g_\star=0\).  Each \(Q_{\star i}\) is nonnegative, so
\(\mathcal E+h\sum_{i<n}Q_{\star i}\geq0\).  Substituting
\eqref{eq:Qij} and cancelling the \(f_i\), \(\norm{g_i}^2\), and
\(x_i\) terms yields
\begin{align*}
0\leq{}&2nh(f_\star-f_n)
 +2\ip{x_0-x_n}{x_0-x_\star}
 -\norm{x_n-x_0}^2
 -\frac{1-\eta}{\eta^2}\norm{g_n}^2\\
={}&\frac{2(1-\eta)}{\eta}(f_\star-f_n)
 +\norm{x_0-x_\star}^2-\norm{x_n-x_\star}^2
 -\frac{1-\eta}{\eta^2}\norm{g_n}^2.
\end{align*}
Multiplying by \(\eta^2/2\) and rearranging gives
\eqref{eq:intro-main-potential}.  Finally,
\(\eta=(1+nh)^{-1}=r^n=(h-1)^n\), which is exactly the algebraic rate
condition in the definition of \(s\)-composability.
\end{proof}

Thus the rest of the paper proves \eqref{eq:scalar-target} by an explicit
nonnegative interpolation certificate.

\section{The interval interpolation certificate}
\label{sec:certificate}

This section constructs the certificate.  Positivity of all scalar weights is
proved separately in \Cref{sec:positivity}.

\subsection{Trajectory matrix and interval modes}

Index matrices by \(\{0,1,\ldots,n\}\), and define the lower-triangular
trajectory matrix \(B\in\R^{(n+1)\times(n+1)}\) by
\begin{equation}
 B_{ij}=\begin{cases}
   h,&j<i,\\
   1,&j=i,\\
   0,&j>i.
 \end{cases}                                         \label{eq:B-def}
\end{equation}
If \(g=(g_0,\ldots,g_n)^\top\), then
\begin{equation}
 (Bg)_i=g_i+h\sum_{j<i}g_j=g_i+x_0-x_i.             \label{eq:Bg}
\end{equation}

For integers \(0\leq a<b\leq n\), define a matrix
\(C^{[a,b]}\) supported on the interval \([a,b]\).  Put
\(\ell=b-a\), and for a row \(i\in[a,b]\) write \(s=i-a\).  Its entries are
\begin{equation}
 C^{[a,b]}_{ij}=\begin{cases}
 h(s+1)-1,&a\leq i<j\leq b,\\
 1+h(\ell-s),&a\leq j<i\leq b,\\
 -\bigl[(\ell-s)(h(s+1)-1)+s(1+h(\ell-s))\bigr],&i=j,\\
 0,&\text{otherwise}.
 \end{cases}                                         \label{eq:interval-mode}
\end{equation}
The diagonal was chosen so that \(C^{[a,b]}\one=0\).

\begin{lemma}[An interval mode diagonalizes]
\label{lem:mode}
Let \(H^{[a,b]}=C^{[a,b]}B+B^\top(C^{[a,b]})^\top\).  Then
\(H^{[a,b]}\) is diagonal.  At the index \(i=a+s\), put
\(t=b-i=\ell-s\).  Its diagonal entry is
\begin{equation}
 \delta(s,t)=2\bigl[r^2(s+1)t-s(t+1)\bigr].         \label{eq:delta}
\end{equation}
Moreover, the cumulative column sum through the cut after \(i\) is zero if
the interval does not cross the cut, and otherwise is
\begin{equation}
 \sigma(s,t)=(s+1)t\left[\alpha+\frac h2(t-s-1)\right].
                                                               \label{eq:sigma}
\end{equation}
\end{lemma}

\begin{proof}
All nonzero off-diagonal entries of \(C^{[a,b]}\) are positive.  Put
\[
 u_s=h(s+1)-1=hs+r,\qquad v_s=1+h(\ell-s).
\]
For two local indices \(s<q\), the definition of \(B\) gives
\[
 (C^{[a,b]}B)_{sq}
 =u_s\bigl[1+h(\ell-q)\bigr]=u_sv_q.
\]
Using the zero row sum in row \(q\),
\[
 (C^{[a,b]}B)_{qs}
 =v_q+h\sum_{j>s}C^{[a,b]}_{qj}
 =v_q-h(s+1)v_q=-v_qu_s.
\]
Thus every off-diagonal entry of
\(C^{[a,b]}B+B^\top(C^{[a,b]})^\top\) vanishes.
On the diagonal,
\begin{align*}
 (C^{[a,b]}B)_{ss}
 &=-(\ell-s)u_s-sv_s+h(\ell-s)u_s\\
 &=r(\ell-s)\bigl(h(s+1)-1\bigr)-s\bigl(1+h(\ell-s)\bigr)\\
 &=r^2(s+1)t-s(t+1),
\end{align*}
which proves \eqref{eq:delta} after symmetrization.

For the cut after local index \(s\), the cumulative column sum is the flow
from right to left minus the flow from left to right.  Directly,
\begin{align*}
 & (s+1)\sum_{q=0}^{t-1}(1+hq)
 -t\sum_{q=0}^{s}\bigl(h(q+1)-1\bigr)\\
 &\quad=(s+1)t\left(1+\frac h2(t-1)
             -\frac h2(s+2)+1\right)\\
 &\quad=(s+1)t\left[\alpha+\frac h2(t-s-1)\right],
\end{align*}
because \(\alpha=2-h\).  This is \eqref{eq:sigma}.
\end{proof}

\subsection{The scalar weights}

The certificate uses the \(n\) suffix modes \([a,n]\) and the \(n-1\)
adjacent modes \([i,i+1]\).  We now define their weights.

For \(i=0,\ldots,n-2\), set
\begin{align}
 \mathsf Z_i
   &=\frac{2hi+3+r-\alpha r^{-2i-2}}{h^2},
 &\mathsf Y_i&=\frac{\mathsf Z_i}{n-i},             \label{eq:ZY-def}\\
 \mathsf W_i
   &=\mathsf Y_i-3\mathsf Y_{i-1}
       +3\mathsf Y_{i-2}-\mathsf Y_{i-3},           \label{eq:W-def}
\end{align}
where \(\mathsf Y_j=0\) for \(j<0\).  With \(p_{-1}=0\), define
\begin{equation}
 p_i=\frac{\mathsf W_i+(n-i+4)p_{i-1}}{n-i-1},
 \qquad i=0,\ldots,n-2.                             \label{eq:p-recurrence}
\end{equation}
For \(i=0,\ldots,n-2\), define
\begin{equation}
 q_i=\frac{i+1-\displaystyle\sum_{a=0}^{i}
             p_a\sigma(i-a,n-i)}{\alpha}.           \label{eq:q-def}
\end{equation}
The final suffix weight is
\begin{equation}
 p_{n-1}=\frac{n-\displaystyle\sum_{a=0}^{n-2}p_a(n-a)
       \left[\alpha-\frac h2(n-1-a)\right]}{\alpha}.
                                                               \label{eq:p-last}
\end{equation}
Empty ranges have their usual meanings.  Thus for \(n=1\), there are no
\(q_i\), \eqref{eq:p-recurrence} is empty, and \(p_0=1/\alpha\).

Define
\begin{equation}
 C=\sum_{a=0}^{n-1}p_aC^{[a,n]}
    +\sum_{i=0}^{n-2}q_iC^{[i,i+1]},
 \qquad H=CB+B^\top C^\top.                         \label{eq:C-H}
\end{equation}

\begin{lemma}[Exact coefficient matching]
\label{lem:matching}
The matrix \(C\) has zero row sums and column-sum vector
\begin{equation}
 b=(\underbrace{1,\ldots,1}_{n\ \mathrm{entries}},-n)^\top.   \label{eq:b}
\end{equation}
The matrix \(H\) is diagonal, with
\begin{equation}
 H_{ii}=\alpha\quad(0\leq i<n),
 \qquad H_{nn}=-n(1+R).                              \label{eq:H-target}
\end{equation}
\end{lemma}

\begin{proof}
Every interval mode has zero row sums and a diagonal symmetric product by
\Cref{lem:mode}; the same is true of their linear combination.

It remains to match the column sums and diagonal.  A zero-row-sum matrix has
column sums \(b\) if and only if its cumulative column sums through cuts
\(0,\ldots,n-1\) equal \(1,2,\ldots,n\).  At cut \(i\leq n-2\),
the suffix modes contribute
\[
 A_i\coloneqq\sum_{a=0}^{i}p_a\sigma(i-a,n-i),
\]
and the sole crossing adjacent mode contributes \(\alpha q_i\), because
\(\sigma(0,1)=\alpha\).  Equation \eqref{eq:q-def} gives
\(A_i+\alpha q_i=i+1\).  At the final cut, the suffix \([a,n]\) contributes
\[
 \sigma(n-1-a,1)=(n-a)
      \left[\alpha-\frac h2(n-1-a)\right].
\]
Equation \eqref{eq:p-last} makes their sum equal to \(n\).  This proves
\eqref{eq:b}.

We next match \(H_{ii}\) for \(i\leq n-2\).  Put \(q_{-1}=0\).  From
\eqref{eq:delta}, the diagonal at \(i\) is
\begin{equation}
 \sum_{a=0}^{i}p_a\delta(i-a,n-i)+2r^2q_i-2q_{i-1}.
                                                               \label{eq:diag-before}
\end{equation}
Multiply \eqref{eq:diag-before} by \(\alpha/2\), and use
\(\alpha q_i=i+1-A_i\) and
\(\alpha q_{i-1}=i-A_{i-1}\).  With \(s=i-a\), \(t=n-i\), and the
convention \(\sigma(-1,t+1)=0\), this gives
\begin{align}
 \frac{\alpha}{2}H_{ii}
 &=r^2(i+1)-i \notag\\
 &\quad+\sum_{a=0}^{i}p_a
 \left\{\alpha\bigl[r^2(s+1)t-s(t+1)\bigr]
       -r^2\sigma(s,t)+\sigma(s-1,t+1)\right\}.    \label{eq:diag-eliminated}
\end{align}
Substitution of \eqref{eq:sigma} into the braces gives
\begin{align*}
 &\alpha\bigl[r^2(s+1)t-s(t+1)\bigr]\\
 &\quad-r^2(s+1)t\left[\alpha+\frac h2(t-s-1)\right]
 +s(t+1)\left[\alpha+\frac h2(t-s+1)\right]\\
 &=\frac h2\left\{
 s(t+1)(t-s+1)-r^2(s+1)t(t-s-1)\right\}.
\end{align*}
Thus the braces equal \((h/2)G(s,t)\), where
\begin{equation}
 G(s,t)=s(t+1)(t-s+1)-r^2(s+1)t(t-s-1).             \label{eq:G}
\end{equation}
Setting \(H_{ii}=\alpha\) in \eqref{eq:diag-eliminated} is therefore
equivalent to
\[
 \sum_{a=0}^{i}p_aG(i-a,n-i)
 =\frac{\alpha^2-2r^2(i+1)+2i}{h}.
\]
Since \(1-r^2=\alpha h\), the right side is \(c_i\), where
\begin{equation}
 c_i=2\alpha i+\frac{\alpha^2-2r^2}{h}.             \label{eq:ci}
\end{equation}
Thus \eqref{eq:diag-before} equals \(\alpha\) exactly when
\begin{equation}
 \sum_{a=0}^{i}p_aG(i-a,n-i)=c_i.                  \label{eq:G-system}
\end{equation}

Here is the complete reduction of \eqref{eq:G-system} to
\eqref{eq:p-recurrence}.  Define
\[
 K(s,t)=(s+1)t(t-s-1).
\]
Then
\[
 G(s,t)=K(s-1,t+1)-r^2K(s,t).
\]
If
\[
 \widehat Z_i=\sum_{a=0}^{i}p_aK(i-a,n-i),
 \qquad \widehat Z_{-1}=0,
\]
the system \eqref{eq:G-system} is
\begin{equation}
 \widehat Z_{i-1}-r^2\widehat Z_i=c_i.              \label{eq:Z-recurrence}
\end{equation}
Extending the displayed formula for \(\mathsf Z_i\) to \(i=-1\) gives
\(\mathsf Z_{-1}=0\).  A direct substitution shows that the unique
solution of \eqref{eq:Z-recurrence} is
\(\widehat Z_i=\mathsf Z_i\) from \eqref{eq:ZY-def}: indeed,
\begin{align*}
 \mathsf Z_{i-1}-r^2\mathsf Z_i
 &=\frac{2h(i-1)+3+r-\alpha r^{-2i}}{h^2}
   -r^2\frac{2hi+3+r-\alpha r^{-2i-2}}{h^2}\\
 &=2\alpha i+\frac{\alpha^2-2r^2}{h}=c_i.
\end{align*}
Moreover, define \(\widehat Y_i=\widehat Z_i/(n-i)\).  Then
\begin{equation}
 \widehat Y_i
 =\sum_{a=0}^{i}p_a(i-a+1)(n-2i+a-1).              \label{eq:Y-convolution}
\end{equation}
Set \(\widehat Y_j=0\) for \(j<0\).  Taking the third backward difference
of \eqref{eq:Y-convolution} cancels the quadratic kernel for every
\(a\leq i-2\).  The two boundary coefficients are
\begin{align*}
 [p_i]:&\quad n-i-1,\\
 [p_{i-1}]:&\quad
 2(n-i-2)-3(n-i)=-(n-i+4).
\end{align*}
Consequently,
\[
 \Delta^3\widehat Y_i
 =(n-i-1)p_i-(n-i+4)p_{i-1}
 =\mathsf W_i=\Delta^3\mathsf Y_i,
\]
where the middle equality is precisely \eqref{eq:p-recurrence}.
Both sequences vanish at the three preceding indices
\(-1,-2,-3\), so induction on the third-difference identity gives
\(\widehat Y_i=\mathsf Y_i\) for every \(0\leq i\leq n-2\).
Therefore \(\widehat Z_i=\mathsf Z_i\), the system
\eqref{eq:G-system} holds, and \(H_{ii}=\alpha\) for
\(0\leq i\leq n-2\).

Only the last two diagonal entries remain.  Let
\(v_i=(-r)^i\).  The geometric sum and \(h=1+r\) give \(Bv=\one\).
Since \(C\one=0\),
\begin{equation}
 \sum_{i=0}^{n}H_{ii}r^{2i}=v^\top Hv=0.            \label{eq:weighted-diag}
\end{equation}
Also, using \(\one^\top C=b^\top\),
\begin{equation}
 \sum_{i=0}^{n}H_{ii}=\one^\top H\one
 =2b^\top B\one.                                   \label{eq:sum-diag}
\end{equation}
The target values in \eqref{eq:H-target} satisfy both identities.  Indeed,
\begin{align*}
 \alpha\sum_{i=0}^{n-1}r^{2i}-n(1+R)r^{2n}
 &=\frac{R^2-1}{hR^2}-\frac{n(1+R)}{R^2}=0,\\
 n\alpha-n(1+R)&=-nh(n+1)
 =2\left(\sum_{i=0}^{n-1}(1+hi)-n(1+hn)\right).
\end{align*}
The first equality uses \(R-1=nh\).  The determinant of the two equations for
\((H_{n-1,n-1},H_{nn})\) is
\(r^{2(n-1)}(r^2-1)\neq0\), so the two remaining entries are uniquely the
target values.
\end{proof}

\subsection{The interpolation sum}

We can now prove the conjecture, conditional only on the positivity result of
the next section.

\begin{theorem}[Constant-step \(s\)-composability]
\label{thm:main}
For every \(n\geq1\), let \(r\in(0,1)\) solve \eqref{eq:root}, and put
\(h=1+r\), \(\eta=r^n\).  Then the constant length-\(n\) schedule
\((h,\ldots,h)\) is \(s\)-composable with rate \(\eta\).
\end{theorem}

\begin{proof}
By \Cref{thm:weights-positive}, the weights \(p_a,q_i\) in
\eqref{eq:p-recurrence}--\eqref{eq:p-last} are positive.  Hence every
off-diagonal entry of \(C\) is nonnegative.  Define
\[
 \lambda_{ij}=C_{ji}\quad(i\neq j),
\]
and set \(\lambda_{ii}=0\), since \(Q_{ii}=0\).

Put \(\beta_i=(Bg)_i\).  From \eqref{eq:Bg},
\[
 x_i-x_j=(\beta_j-g_j)-(\beta_i-g_i).
\]
Substitution into \eqref{eq:Qij} gives the exact identity
\begin{equation}
 Q_{ij}=2f_i-2f_j+2\ip{g_j}{\beta_i-\beta_j}
                -\norm{g_i}^2+\norm{g_j}^2.         \label{eq:Q-B}
\end{equation}
Sum \(C_{ji}Q_{ij}\) over all \(i,j\).  The zero row sums eliminate the
terms containing \(\beta_j\), while the column sums are \(b\).  Therefore
\begin{align*}
 \sum_{i,j}C_{ji}Q_{ij}
 &=2\sum_i b_if_i+2g^\top CBg-\sum_i b_i\norm{g_i}^2\\
 &=2\sum_i b_if_i+g^\top Hg-\sum_i b_i\norm{g_i}^2.
\end{align*}
Using \eqref{eq:b} and \eqref{eq:H-target}, this is exactly
\begin{equation}
 2\sum_{i=0}^{n-1}(f_i-f_n)
 -r\sum_{i=0}^{n-1}\norm{g_i}^2-nR\norm{g_n}^2.
\end{equation}
Every \(Q_{ij}\geq0\) and every off-diagonal multiplier is nonnegative, so
the expression is nonnegative.  This is \eqref{eq:scalar-target}, and
\Cref{lem:reduction} completes the proof.
\end{proof}

\section{Positivity of every certificate weight}
\label{sec:positivity}

This section proves the only inequality statement needed by the certificate.
All quantities below are the explicit scalars defined in
\eqref{eq:parameters} and \eqref{eq:ZY-def}--\eqref{eq:p-last}.  We use
\begin{equation}
 u=r^2,\qquad d=1-u=\alpha h,\qquad R=r^{-n}=1+nh.
 \label{eq:positivity-notation}
\end{equation}

\subsection{Bounds on the balancing root}

\begin{lemma}[Three root bounds]
\label{lem:root-bounds}
For \(n\geq2\), \(r\geq\tfrac12\).  For \(n\geq3\),
\begin{equation}
 nd>2,
 \qquad
 (n+4)d>4.
 \label{eq:root-bounds}
\end{equation}
\end{lemma}

\begin{proof}
Let
\[
 F_n(x)=x^n\bigl(1+n(1+x)\bigr).
\]
This function is strictly increasing for \(x>0\), and \(F_n(r)=1\).
For \(n\geq2\),
\[
 F_n(1/2)=2^{-n}\left(1+\frac{3n}{2}\right)\leq1.
\]
The last inequality is equality at \(n=2\); if
\(1+3n/2\leq2^n\), then
\(1+3(n+1)/2\leq2^{n+1}\), because
\(2(1+3n/2)-(1+3(n+1)/2)=(3n-1)/2>0\).
Thus \(r\geq1/2\).

To prove \(nd>2\), first suppose \(n\geq4\) and set
\(s=(1-2/n)^{1/2}\).  The function
\[
 a(x)=\left(1-\frac2x\right)^{x/2},\qquad x\geq4,
\]
is increasing.  Indeed,
\[
 \frac{d}{dx}\log a(x)
 =\frac12\log\left(1-\frac2x\right)+\frac1{x-2}\geq0,
\]
where \(\log(1-y)\geq-y/(1-y)\) for \(0<y<1\) was used with
\(y=2/x\).  Hence \(s^n=a(n)\geq a(4)=1/4\), and
\[
 F_n(s)=s^n\bigl(1+n(1+s)\bigr)
 >\frac{n+1}{4}>1.
\]
For \(n=3\), \(s=1/\sqrt3\) and
\[
 F_3(s)=\frac{4+\sqrt3}{3\sqrt3}>1.
\]
Monotonicity of \(F_n\) therefore gives \(r<s\), or \(nd>2\).

For the second inequality in \eqref{eq:root-bounds}, set
\(s=(n/(n+4))^{1/2}\).  If \(n\geq5\), then \(s>1/2\) and
\[
 s^{-n}=\left(1+\frac4n\right)^{n/2}<e^2<8
 <1+\frac{3n}{2}<1+n(1+s).
\]
Here \(\log(1+z)<z\) proves the first strict inequality.  For a
fully elementary bound on the constant, \(k!\geq6\,4^{k-3}\) for
\(k\geq3\), and hence
\[
 e^2=1+2+2+\sum_{k\geq3}\frac{2^k}{k!}
 \leq5+\frac83<8.
\]
Thus \(F_n(s)>1\).  The two remaining cases are direct:
\[
 F_4(1/\sqrt2)=\frac{5+2\sqrt2}{4}>1,
\]
whereas for \(n=3\), writing \(s=\sqrt{3/7}\) gives
\[
 F_3(s)=\frac{12s}{7}+\frac{27}{49}>1
\]
because \(s>11/42\).  Again \(r<s\), which is
\((n+4)d>4\).
\end{proof}

\subsection{Suffix weights}

The recurrence for the suffix weights admits the following form.  Its
finite-difference derivation, including both boundary indices, is given in
\Cref{app:suffix-closed-form}.

\begin{lemma}[Closed form for the nonterminal suffix weights]
\label{lem:p-closed-form}
For \(n\geq2\),
\begin{equation}
 p_0=\frac{r^2+2r-1}{nr^2(n-1)h}.                 \label{eq:p0-closed}
\end{equation}
For \(1\leq i\leq n-2\), put \(t=n-i\), so \(2\leq t\leq n-1\).  Then
\begin{equation}
 p_i=\frac{6\mathsf A_t-\mathsf D_tR^2u^{t-1}}
 {h^3(t-1)t(t+1)(t+2)(t+3)},                      \label{eq:p-closed}
\end{equation}
where
\begin{align}
 \mathsf A_t
 &=2h^2n(2n-t)+h\bigl(2n(r+5)-t(r+3)\bigr)
       +r^2+4r+7,                                  \label{eq:At}\\
 \mathsf D_t
 &=t^3d^3+6t^2d^2+td(18-6d-d^2)+24-18d.           \label{eq:Dt}
\end{align}
\end{lemma}

\begin{lemma}[All suffix weights are positive]
\label{lem:p-positive}
For every \(n\geq1\), all of the weights \(p_0,\ldots,p_{n-1}\) are
strictly positive.
\end{lemma}

\begin{proof}
For \(n=1\), the sole weight is \(p_0=1/\alpha>0\).  Suppose first that
\(n\geq2\).  By \Cref{lem:root-bounds}, \(r\geq1/2\), and hence
\(r^2+2r-1\geq1/4>0\).  Formula \eqref{eq:p0-closed} gives \(p_0>0\).

It remains to prove positivity of the numerator in \eqref{eq:p-closed};
this range is nonempty only for \(n\geq3\).  Define
\begin{align}
 \mathsf E_t
 &\coloneqq\frac{\mathsf D_t-u\mathsf D_{t+1}}d  \notag\\
 &=t^3d^3+3t^2d^3+3t^2d^2+2td^3+3td^2+6td+6.
                                                               \label{eq:Et}
\end{align}
Every term in \(\mathsf D_t\) and \(\mathsf E_t\) is positive: in
\eqref{eq:Dt}, \(18-6d-d^2\geq11\) and \(24-18d\geq6\).
A direct expansion, grouped into positive factors, gives
\begin{equation}
 (4-3d)\mathsf E_t-\mathsf D_t
 =3d(1-d)t\bigl[d^2(t+1)(t+2)+2d(t+1)+2\bigr]>0.  \label{eq:DE-gap}
\end{equation}

Put
\begin{equation}
 \gamma=2h^2n+h(r+3).
\end{equation}
Rearranging \eqref{eq:At} gives
\begin{equation}
 \mathsf A_t=\gamma(2n-t)+4hn+r^2+4r+7,
 \qquad \mathsf A_{t+1}=\mathsf A_t-\gamma.       \label{eq:A-linear}
\end{equation}
For every \(2\leq t\leq n-2\), equations
\eqref{eq:Et}--\eqref{eq:A-linear} imply
\begin{align*}
 &\mathsf A_{t+1}\mathsf D_t
       -u\mathsf A_t\mathsf D_{t+1}\\
 &\quad=\mathsf A_t(\mathsf D_t-u\mathsf D_{t+1})
            -\gamma\mathsf D_t\\
 &\quad=\mathsf A_t d\mathsf E_t-\gamma\mathsf D_t\\
 &\quad>\gamma\mathsf E_t
       \bigl(d(2n-t)-(4-3d)\bigr)\\
 &\quad=\gamma\mathsf E_t
       \bigl(d(2n-t+3)-4\bigr)>0.
\end{align*}
The first strict inequality uses
\(\mathsf A_t>\gamma(2n-t)\) and
\(\mathsf D_t<(4-3d)\mathsf E_t\).  For the last one,
\(t\leq n-2\) gives \(2n-t+3\geq n+5\), while
\((n+4)d>4\) by \Cref{lem:root-bounds}.  Consequently,
\begin{equation}
 \frac{\mathsf A_{t+1}}{\mathsf D_{t+1}u^t}
 >\frac{\mathsf A_t}{\mathsf D_tu^{t-1}}.         \label{eq:ratio-increasing}
\end{equation}

It is therefore enough to check the endpoint \(t=2\).  Let
\[
 \mathsf K=d^3+2d^2+3d+4.
\]
Substitution in \eqref{eq:At}--\eqref{eq:Dt}, using \(R=1+nh\), gives
\begin{equation}
 \mathsf A_2=4R^2-2hR-h^2,
 \qquad \mathsf D_2=6\mathsf K.                  \label{eq:t2-values}
\end{equation}
The polynomial needed at this endpoint factors as
\begin{equation}
 4-u\mathsf K=d(2-u)(1+d^2)>d=\alpha h.           \label{eq:t2-factor}
\end{equation}
Furthermore,
\begin{equation}
 \alpha R=\alpha+nd>2+\alpha,
 \qquad
 \alpha R-h=nd-2r>0,                              \label{eq:alphaR}
\end{equation}
where \(nd>2\) and \(r<1\) were used.  The second inequality in
\eqref{eq:alphaR} says \(\alpha>h/R\); combining it with the first gives
\(\alpha R>2+h/R\), or
\begin{equation}
 \alpha R^2>2R+h.                                 \label{eq:alphaR2}
\end{equation}
Now \eqref{eq:t2-values}--\eqref{eq:alphaR2} yield
\begin{align*}
 \mathsf A_2-\mathsf K R^2u
 &=R^2(4-u\mathsf K)-h(2R+h)\\
 &>h\bigl(\alpha R^2-2R-h\bigr)>0.
\end{align*}
Thus \(6\mathsf A_2>\mathsf D_2R^2u\).  The monotonicity
\eqref{eq:ratio-increasing} implies
\(6\mathsf A_t>\mathsf D_tR^2u^{t-1}\) throughout
\(2\leq t\leq n-1\), and \eqref{eq:p-closed} gives \(p_i>0\) for
\(1\leq i\leq n-2\).

Finally, every bracket in the definition \eqref{eq:p-last} obeys
\[
 \alpha-\frac h2(n-1-a)
 \leq\alpha-\frac h2=\frac{1-3r}{2}<0.
\]
All already constructed \(p_a\) are positive.  The numerator of
\eqref{eq:p-last} is therefore larger than \(n\), so \(p_{n-1}>0\).
\end{proof}

\subsection{Adjacent weights}

\begin{lemma}[Closed form and positivity of the adjacent weights]
\label{lem:q-positive}
For \(1\leq k\leq n-1\), set \(t=n-k\).  Then
\begin{equation}
 q_{k-1}=
 \frac{(th+2)^2-r^{2k}\bigl((t+2k)h+2\bigr)^2}
 {2h^3r^{2k}t(t+2)}>0.                            \label{eq:q-closed}
\end{equation}
\end{lemma}

\begin{proof}
The finite-sum derivation of \eqref{eq:q-closed} from \eqref{eq:q-def}
is given in \Cref{app:q-closed-form}; it uses only the already defined
weights \(p_0,\ldots,p_{k-1}\), including the special left boundary
\eqref{eq:p0-closed}.

It remains to prove the strict sign.  Put \(A=nh+2\) and define, for
\(0\leq x\leq n\),
\begin{equation}
 \Phi(x)=\log\frac{A-hx}{A+hx}-x\log r.           \label{eq:Phi}
\end{equation}
Both logarithm arguments are positive.  We have \(\Phi(0)=0\), and the
root equation gives
\[
 \Phi(n)=\log\frac{2}{2(1+nh)}-n\log r
 =\log(r^n)-n\log r=0.
\]
For \(0<x<n\),
\begin{equation}
 \Phi''(x)=h^2\left(\frac1{(A+hx)^2}
                    -\frac1{(A-hx)^2}\right)<0.  \label{eq:Phi-second}
\end{equation}
Strict concavity and the two zero endpoint values imply
\(\Phi(k)>0\) for \(1\leq k\leq n-1\).  Since
\(A-hk=th+2\) and \(A+hk=(t+2k)h+2\), this is precisely
\[
 th+2>r^k\bigl((t+2k)h+2\bigr).
\]
The two sides are positive, so the numerator in \eqref{eq:q-closed} is
strictly positive; its denominator is also positive.
\end{proof}

\begin{theorem}[Certificate-weight positivity]
\label{thm:weights-positive}
For every integer \(n\geq1\), every suffix weight \(p_a\) and every
adjacent weight \(q_i\) in \eqref{eq:C-H} is strictly positive.
\end{theorem}

\begin{proof}
Combine \Cref{lem:p-positive,lem:q-positive}.  When \(n=1\), the adjacent
family is empty.
\end{proof}

\section{Consequences and equivalent formulations}
\label{sec:consequences}

\subsection{The result at arbitrary smoothness}

\begin{corollary}[Scaled constant-step theorem]
\label{cor:scaled}
Let \(f:\R^d\to\R\) be convex and \(L\)-smooth, with \(L>0\), and let
\(x_\star\in\argmin f\).  For \(r\) determined by \eqref{eq:root}, run
\begin{equation}
 x_{i+1}=x_i-\frac{1+r}{L}\nabla f(x_i),
 \qquad i=0,\ldots,n-1.                            \label{eq:scaled-gd}
\end{equation}
With \(\eta=r^n=(1+n(1+r))^{-1}\), the terminal iterate satisfies
\begin{align}
 &\frac{1-\eta}{2L}\norm{\nabla f(x_n)}^2
 +\frac{\eta^2L}{2}\norm{x_n-x_\star}^2
 +(\eta-\eta^2)\bigl(f(x_n)-f(x_\star)\bigr)     \notag\\
 &\hspace{40mm}\leq
 \frac{\eta^2L}{2}\norm{x_0-x_\star}^2.          \label{eq:scaled-potential}
\end{align}
\end{corollary}

\begin{proof}
The function \(\widetilde f=f/L\) is \(1\)-smooth and convex, and
\(\nabla\widetilde f=\nabla f/L\).  Thus \eqref{eq:scaled-gd} is exactly
\(x_{i+1}=x_i-h\nabla\widetilde f(x_i)\).  Apply \Cref{thm:main} to
\(\widetilde f\), replace
\(\nabla\widetilde f(x_n)\) by \(\nabla f(x_n)/L\) and
\(\widetilde f(x_n)-\widetilde f(x_\star)\) by
\((f(x_n)-f(x_\star))/L\), and multiply the resulting inequality by
\(L\).
\end{proof}

Since every term on the left of \eqref{eq:scaled-potential} is
nonnegative, the same inequality simultaneously gives
\begin{align}
 f(x_n)-f(x_\star)
 &\leq\frac{\eta L}{2(1-\eta)}\norm{x_0-x_\star}^2
 =\frac{L}{2nh}\norm{x_0-x_\star}^2,              \label{eq:f-bound}\\
 \norm{\nabla f(x_n)}^2
 &\leq\frac{\eta^2L^2}{1-\eta}\norm{x_0-x_\star}^2
 =\frac{\eta L^2}{nh}\norm{x_0-x_\star}^2,       \label{eq:g-bound}\\
 \norm{x_n-x_\star}
 &\leq\norm{x_0-x_\star}.                         \label{eq:x-bound}
\end{align}
The equalities use \(1-\eta=nh\eta\).

\begin{proposition}[Sharpness of the mixed potential]
\label{prop:quadratic-tight}
Inequality \eqref{eq:scaled-potential} is attained by the quadratic
\(f(x)=L\norm{x}^2/2\), with \(x_\star=0\), for every initial point.
\end{proposition}

\begin{proof}
For this function, \eqref{eq:scaled-gd} gives
\(x_{i+1}=(1-h)x_i=-rx_i\), hence
\(x_n=(-r)^nx_0\),
\(\norm{x_n}^2=\eta^2\norm{x_0}^2\),
\(\nabla f(x_n)=Lx_n\), and
\(f(x_n)=L\eta^2\norm{x_0}^2/2\).  The left side of
\eqref{eq:scaled-potential}, divided by
\(L\norm{x_0}^2/2\), becomes
\[
 (1-\eta)\eta^2+\eta^4+(\eta-\eta^2)\eta^2
 =\eta^2,
\]
which is the normalized right side.
\end{proof}

\subsection{A relaxed proximal-point energy inequality}

There is a useful conjugate interpretation of the same theorem.  We give
the complete equivalence because it explains why a gradient-descent
certificate can also be read as a relaxed proximal-point estimate.

Work in the \(L=1\) normalization and translate constants so that
\(x_\star=0\) and \(f(x_\star)=0\).  Let \(f^*\) be the Fenchel conjugate
and define
\begin{equation}
 \psi(g)=f^*(g)-\frac12\norm{g}^2,
 \qquad \varphi=\psi^*.                            \label{eq:psi-phi}
\end{equation}
Because \(f\) is \(1\)-smooth and convex, \(f^*\) is \(1\)-strongly
convex, so \(\psi\) is convex.  At a sampled point, put
\begin{equation}
 g_i=\nabla f(x_i),\qquad y_i=x_i-g_i.              \label{eq:y-def}
\end{equation}
Fenchel reciprocity gives \(x_i\in\partial f^*(g_i)\), and therefore
\[
 y_i=x_i-g_i\in\partial\psi(g_i)
 \quad\Longleftrightarrow\quad
 g_i\in\partial\varphi(y_i).
\]
Since \(x_i-y_i=g_i\), this is equivalent to
\begin{equation}
 y_i=\operatorname{prox}_{\varphi}(x_i),
 \qquad x_{i+1}=x_i-h(x_i-y_i).                    \label{eq:rppa}
\end{equation}
Thus gradient descent corresponds to the relaxed proximal point algorithm
with relaxation \(h\); compare the finite-horizon RPPA analysis of
\citet{wang2025rppa}.

\begin{corollary}[Terminal RPPA energy]
\label{cor:rppa}
Every trajectory \eqref{eq:rppa} arising through \eqref{eq:psi-phi}
satisfies
\begin{equation}
 \frac12\norm{x_n}^2+nh\,\varphi(y_n)
 +nh\left(1+\frac{nh}{2}\right)\norm{g_n}^2
 \leq\frac12\norm{x_0}^2.                        \label{eq:rppa-energy}
\end{equation}
\end{corollary}

\begin{proof}
Fenchel equality at \(g_n=\nabla f(x_n)\) gives
\[
 f(x_n)=\ip{g_n}{x_n}-f^*(g_n)
 =\ip{g_n}{x_n}-\frac12\norm{g_n}^2-\psi(g_n).
\]
Insert this identity in \eqref{eq:intro-main-potential}.  The three
quadratic terms on its left combine as
\begin{align*}
 &\frac{1-\eta}{2}\norm{g_n}^2
 +\frac{\eta^2}{2}\norm{x_n}^2
 +(\eta-\eta^2)
    \left(\ip{g_n}{x_n}-\frac12\norm{g_n}^2\right)\\
 &\qquad=\frac12\norm{\eta x_n+(1-\eta)g_n}^2.
\end{align*}
Since \((1-\eta)/\eta=nh\), division by \(\eta^2\) shows that the
potential inequality is equivalent to
\begin{equation}
 \frac12\norm{x_n+nhg_n}^2-nh\,\psi(g_n)
 \leq\frac12\norm{x_0}^2.                        \label{eq:psi-energy}
\end{equation}
Finally, \(g_n\in\partial\varphi(y_n)\) implies
\(\psi(g_n)+\varphi(y_n)=\ip{g_n}{y_n}\), and
\(y_n=x_n-g_n\).  Expanding the left side of
\eqref{eq:psi-energy} therefore gives
\begin{align*}
 &\frac12\norm{x_n+nhg_n}^2
 -nh\bigl(\ip{g_n}{y_n}-\varphi(y_n)\bigr)\\
 &=\frac12\norm{x_n}^2+nh\ip{g_n}{x_n-y_n}
   +\frac{n^2h^2}{2}\norm{g_n}^2+nh\varphi(y_n)\\
 &=\frac12\norm{x_n}^2+nh\left(1+\frac{nh}{2}\right)
   \norm{g_n}^2+nh\varphi(y_n),
\end{align*}
which is \eqref{eq:rppa-energy}.
\end{proof}

\begin{remark}
The argument above is an equivalence for proximal problems obtained from a
smooth convex \(f\) by \eqref{eq:psi-phi}; it is not asserted here as a
worst-case theorem for an arbitrary independently chosen nonsmooth
\(\varphi\).  This distinction avoids enlarging the proved problem class.
\end{remark}

\section{Conclusion}
\label{sec:conclusion}

We proved that the balanced constant schedule of
\citet[Example~2]{grimmer2025composing} is \(s\)-composable at every
horizon.  The proof is an explicit smooth-convex interpolation certificate:
the multiplier matrix is a positive combination of \(n\) suffix modes and
\(n-1\) adjacent modes, its column sums match the function coefficients,
and its symmetrized trajectory product is exactly diagonal.  The
finite-difference and Abel-sum calculations expose all scalar multipliers,
while the root bounds and a strict-concavity argument prove their positivity.

The interval-mode construction may be useful beyond this one conjecture.
It converts a semidefinite certificate search into two one-dimensional
tasks: matching flow across cuts and matching diagonal entries.  A natural
next question is whether other numerically observed PEP certificates admit
similarly sparse interval decompositions, or whether the suffix/adjacent
basis can be generated systematically from a dual semidefinite solution.

The literature-status assertion and the new proof should be treated
separately.  The former is based on the current primary source and a search
of later citing literature through August 31, 2026; an uncatalogued
concurrent manuscript is always possible.  The latter is self-contained,
but this document is an unrefereed research draft.  Independent checking
and peer review remain necessary before the result should be relied upon as
part of the established literature.

\section{Disclosure}
\label{sec:disclosure}
The proof strategy is produced by OpenAI’s GPT-5.6 Sol Ultra through Codex in
response to prompts from the author. Codex was also used to revise
the exposition and prepare the LaTeX manuscript. The author selected the
problem, directed the interactions and revisions, and is the sole named author. The AI system is acknowledged as a reasoning and writing tool, not
as an author. This disclosure is not a substitute for independent expert
mathematical review.
\bibliographystyle{plainnat}
\bibliography{references}
\appendix
\section{Finite-sum derivations and boundary certificates}
\label{app:finite-sums}

\subsection{Derivation of the suffix-weight closed form}
\label{app:suffix-closed-form}

We derive \Cref{lem:p-closed-form} directly from the recurrence, retaining
the left boundary where the generic formula does not apply.  Let
\[
 F(s)=s(s+1)(s+2)(s+3).
\]
For \(i\leq n-2\), put \(t=n-i\).  Equation
\eqref{eq:p-recurrence} is
\[
 (t-1)p_i-(t+4)p_{i-1}=\mathsf W_i.
\]
Multiplication by \(F(t)\) gives
\[
 (t-1)F(t)p_i=(t+4)F(t)p_{i-1}+F(t)\mathsf W_i.
\]
The coefficient \((t+4)F(t)=t(t+1)(t+2)(t+3)(t+4)\) is exactly the
coefficient \((t+1-1)F(t+1)\) attached to \(p_{i-1}\) at the preceding
index.  Since \(p_{-1}=0\), induction therefore yields
\begin{equation}
 p_i=\frac{\mathsf S_i}{(t-1)t(t+1)(t+2)(t+3)},
 \qquad
 \mathsf S_i=\sum_{j=0}^{i}F(n-j)\mathsf W_j.       \label{eq:Si}
\end{equation}

We now evaluate \(\mathsf S_i\).  Recall that
\(\mathsf W_j=\mathsf Y_j-3\mathsf Y_{j-1}
+3\mathsf Y_{j-2}-\mathsf Y_{j-3}\) and
\(\mathsf Y_\ell=\mathsf Z_\ell/(n-\ell)\).  For an interior index
\(\ell\leq i-3\), the coefficient of \(\mathsf Y_\ell\) in
\eqref{eq:Si} is, with \(s=n-\ell\),
\[
 F(s)-3F(s-1)+3F(s-2)-F(s-3)=24s.
\]
After division by \(s\), this contributes \(24\mathsf Z_\ell\).
At the three right boundary indices, direct division by
\(n-\ell\) gives
\begin{align*}
 [\mathsf Z_{i-2}]:\;&
 \frac{F(t+2)-3F(t+1)+3F(t)}{t+2}
 =(t+3)(t^2-3t+8),\\
 [\mathsf Z_{i-1}]:\;&
 \frac{F(t+1)-3F(t)}{t+1}
 =-2(t-2)(t+2)(t+3),\\
 [\mathsf Z_i]:\;&
 \frac{F(t)}{t}=(t+1)(t+2)(t+3).
\end{align*}
Consequently, for \(i\geq2\),
\begin{align}
 \mathsf S_i
 ={}&24\sum_{j=0}^{i-3}\mathsf Z_j
 +(t+3)(t^2-3t+8)\mathsf Z_{i-2} \notag\\
 &-2(t-2)(t+2)(t+3)\mathsf Z_{i-1}
 +(t+1)(t+2)(t+3)\mathsf Z_i.                     \label{eq:Si-expanded}
\end{align}

The remaining sum is geometric.  For every \(L\geq0\),
\begin{align}
 \sum_{j=0}^{L}\mathsf Z_j
 &=\frac{1}{h^2}\left[
 hL(L+1)+(L+1)(3+r)
 -\alpha\sum_{j=0}^{L}r^{-2j-2}\right] \notag\\
 &=\frac{hL(L+1)+(L+1)(3+r)
 -(r^{-2L-2}-1)/h}{h^2}.                          \label{eq:sum-Z}
\end{align}
For the second equality we used
\[
 \alpha\sum_{j=0}^{L}r^{-2j-2}
 =\frac{\alpha}{1-r^2}(r^{-2L-2}-1)
 =\frac{r^{-2L-2}-1}{h}.
\]

For transparency, we now perform the collection that leads to
\eqref{eq:p-closed}.  Define
\begin{align*}
 c_2&=(t+3)(t^2-3t+8),\\
 c_1&=-2(t-2)(t+2)(t+3),\\
 c_0&=(t+1)(t+2)(t+3).
\end{align*}
With \(i=n-t\), the individual terms in \eqref{eq:Si-expanded} are
\begin{equation}
 \mathsf Z_{i-\ell}
 =\frac{2h(n-t-\ell)+3+r
       -\alpha R^2r^{2t+2\ell-2}}{h^2},
 \qquad \ell=0,1,2,                               \label{eq:Z-shifted}
\end{equation}
and \eqref{eq:sum-Z} gives
\begin{align}
 h^3\sum_{j=0}^{i-3}\mathsf Z_j
 ={}&h^2(i-3)(i-2)+h(i-2)(3+r) \notag\\
 &+1-R^2r^{2t+4}.                                 \label{eq:sum-Z-shifted}
\end{align}
When \(i=2\), the sum is empty and the right side is also zero, because
\(R^2r^{2t+4}=R^2r^{2n}=1\).  Thus
\eqref{eq:sum-Z-shifted} also covers that boundary.

Substitution of \eqref{eq:Z-shifted} and
\eqref{eq:sum-Z-shifted} into \eqref{eq:Si-expanded} gives the entirely
expanded identity
\begin{align}
 h^3\mathsf S_i
={}&24\!\left[h^2(i-3)(i-2)+h(i-2)(3+r)
               +1-R^2r^{2t+4}\right] \notag\\
 &+c_2\!\left[h\{2h(i-2)+3+r\}
               -dR^2r^{2t+2}\right] \notag\\
 &+c_1\!\left[h\{2h(i-1)+3+r\}
               -dR^2r^{2t}\right] \notag\\
 &+c_0\!\left[h(2hi+3+r)
               -dR^2r^{2t-2}\right].              \label{eq:Si-fully-expanded}
\end{align}
The coefficient of \(-R^2r^{2t-2}\) on the right is
\begin{align*}
 &24r^6+dc_2r^4+dc_1r^2+dc_0\\
 &\quad=24u^3+d(t+3)(t^2-3t+8)u^2\\
 &\qquad\quad-2d(t-2)(t+2)(t+3)u
       +d(t+1)(t+2)(t+3)\\
 &\quad=t^3d^3+6t^2d^2+td(18-6d-d^2)+24-18d
 =\mathsf D_t,
\end{align*}
where the penultimate line follows by replacing \(u\) with \(1-d\)
and collecting powers of \(d\).  The terms in
\eqref{eq:Si-fully-expanded} not containing \(R^2\), after replacing
\(i\) by \(n-t\), are
\begin{align*}
 &24\!\left[h^2(n-t-3)(n-t-2)
       +h(n-t-2)(3+r)+1\right]\\
 &\quad+c_2h\{2h(n-t-2)+3+r\}\\
 &\quad+c_1h\{2h(n-t-1)+3+r\}
       +c_0h\{2h(n-t)+3+r\}\\
 &=6\left[
 2h^2n(2n-t)+h\{2n(r+5)-t(r+3)\}+r^2+4r+7
 \right]
 =6\mathsf A_t.
\end{align*}
This proves
\[
 \mathsf S_i=\frac{6\mathsf A_t-\mathsf D_tR^2r^{2t-2}}{h^3},
\]
and \eqref{eq:Si} proves \eqref{eq:p-closed} for \(i\geq2\).

At \(i=1\), the interior sum and the \(\mathsf Z_{i-2}\) term are absent.
Writing \(t=n-1\), the exact boundary expression is
\[
 \mathsf S_1
 =-2(t-2)(t+2)(t+3)\mathsf Z_0
 +(t+1)(t+2)(t+3)\mathsf Z_1.
\]
Inserting
\[
 \mathsf Z_0=\frac{3+r-\alpha r^{-2}}{h^2},
 \qquad
 \mathsf Z_1=\frac{2h+3+r-\alpha r^{-4}}{h^2},
\]
and using \(n=t+1\) and \(R^2r^{2t-2}=r^{-4}\) gives the same identity
\[
 h^3\mathsf S_1=6\mathsf A_t-\mathsf D_tR^2r^{2t-2}.
\]
Thus \eqref{eq:p-closed} also holds at \(i=1\).

Finally, at \(i=0\), \(\mathsf W_0=\mathsf Y_0=\mathsf Z_0/n\), so
\begin{align*}
 p_0
 &=\frac{\mathsf Z_0}{n(n-1)}
 =\frac{3+r-\alpha r^{-2}}{h^2n(n-1)}\\
 &=\frac{r^2+2r-1}{nr^2(n-1)h}.
\end{align*}
The last equality follows after multiplying the numerators by \(r^2\):
\[
 r^2(3+r)-\alpha
 =r^3+3r^2+r-1
 =h(r^2+2r-1).
\]
This proves \eqref{eq:p0-closed} and explains why the generic expression
\eqref{eq:p-closed} must not be used at \(i=0\).

\subsection{Derivation of the adjacent-weight closed form}
\label{app:q-closed-form}

Assume \(n\geq2\).  We give a boundary-safe Abel-sum derivation of
\eqref{eq:q-closed}.  From
\eqref{eq:Y-convolution} and
\(\widehat Z_i=\mathsf Z_i\), we have
\begin{equation}
 \mathsf Y_i
 =\sum_{a=0}^{i}p_a(i-a+1)(n-2i+a-1).             \label{eq:Y-again}
\end{equation}
For \(0\leq i\leq n-2\), define
\[
 \mathsf M_i=\sum_{a=0}^{i}(i-a+1)p_a,
 \qquad
 \mathsf H_i=\sum_{j=0}^{i}\mathsf M_j,
 \qquad \mathsf H_{-1}=0.
\]
If \(m=i-a+1\), then
\[
 2\mathsf H_{i-1}
 =\sum_{a=0}^{i}m(m-1)p_a.
\]
The factor multiplying \(p_a\) in \eqref{eq:Y-again} is
\(m(n-i-m)\); hence
\begin{align}
 \mathsf Y_i
 &=(n-i)\mathsf M_i
   -\sum_{a=0}^{i}m^2p_a \notag\\
 &=(n-i-1)\mathsf M_i-2\mathsf H_{i-1}.            \label{eq:Y-MH}
\end{align}
Since \(\mathsf M_i=\mathsf H_i-\mathsf H_{i-1}\), putting
\(\tau_i=n-i-1\) in \eqref{eq:Y-MH} gives
\[
 \tau_i\mathsf H_i=(\tau_i+2)\mathsf H_{i-1}
                    +\mathsf Y_i.
\]
Induction, starting at
\(\mathsf H_0=\mathsf Y_0/(n-1)\), now gives, for
\(0\leq m\leq n-2\),
\begin{equation}
 \mathsf H_m
 =\frac{\sum_{j=0}^{m}(n-j)\mathsf Y_j}
        {(n-m-1)(n-m)}
 =\frac{\sum_{j=0}^{m}\mathsf Z_j}
        {(n-m-1)(n-m)}.                            \label{eq:H-Abel}
\end{equation}
To verify the induction step explicitly, multiply the formula for
\(\mathsf H_{m-1}\) by \(n-m+1\), add \(\mathsf Y_m\), and divide by
\(n-m-1\):
\[
 \mathsf H_m
 =\frac{\sum_{j=0}^{m-1}(n-j)\mathsf Y_j
          +(n-m)\mathsf Y_m}
        {(n-m-1)(n-m)}.
\]

Fix \(1\leq k\leq n-1\) and put \(t=n-k\).  Taking \(i=k-1\) in
\eqref{eq:Y-MH} and using \eqref{eq:H-Abel} gives
\begin{equation}
 \mathsf M_{k-1}
 =\frac{\mathsf Z_{k-1}}{t(t+1)}
 +\frac{2\sum_{j=0}^{k-2}\mathsf Z_j}
        {t(t+1)(t+2)},                             \label{eq:M-explicit}
\end{equation}
where the sum is empty when \(k=1\).

Let \(\mathcal C_k\) be the suffix contribution to the cut after \(k-1\):
\[
 \mathcal C_k
 =\sum_{a=0}^{k-1}p_a\sigma(k-1-a,t+1).
\]
With \(m=k-a\), formula \eqref{eq:sigma} gives
\[
 \mathcal C_k
 =(t+1)\sum_{a=0}^{k-1}p_am
 \left[\alpha+\frac h2(t+1-m)\right].
\]
The first moment is \(\mathsf M_{k-1}\), while
\[
 \sum_{a=0}^{k-1}p_am(t+1-m)
 =t\mathsf M_{k-1}-2\mathsf H_{k-2}
 =\mathsf Y_{k-1}
\]
by \eqref{eq:Y-MH}.  Since
\((t+1)\mathsf Y_{k-1}=\mathsf Z_{k-1}\), it follows that
\begin{equation}
 \mathcal C_k
 =(t+1)\alpha\mathsf M_{k-1}
   +\frac h2\mathsf Z_{k-1}.                       \label{eq:Ck}
\end{equation}

The two finite sums in \eqref{eq:M-explicit} are
\begin{align}
 \mathsf Z_{k-1}
 &=\frac{2hk+\alpha(1-r^{-2k})}{h^2},              \label{eq:Z-k}\\
 \sum_{j=0}^{k-2}\mathsf Z_j
 &=\frac{(k-1)(hk+\alpha)
       -(r^{-2k+2}-1)/h}{h^2}.                    \label{eq:sum-Z-k}
\end{align}
Equation \eqref{eq:Z-k} is obtained from \eqref{eq:ZY-def} using
\(-2h+3+r=\alpha\).  For \(k\geq2\), equation
\eqref{eq:sum-Z-k} is \eqref{eq:sum-Z} with \(L=k-2\).  For \(k=1\),
the sum is empty and the right side of \eqref{eq:sum-Z-k} is zero
directly, so the formula remains valid at that boundary.

Substituting \eqref{eq:M-explicit}, \eqref{eq:Z-k}, and
\eqref{eq:sum-Z-k} into \eqref{eq:Ck}, and putting all terms over the
common denominator \(2h^3r^{2k}t(t+2)\), gives
\begin{align}
 k-\mathcal C_k
 =\frac{\alpha\left[
 (th+2)^2-r^{2k}\bigl((t+2k)h+2\bigr)^2\right]}
 {2h^3r^{2k}t(t+2)}.                              \label{eq:cut-residual}
\end{align}
For completeness, the substitution before collecting is
\begin{align*}
 k-\mathcal C_k
 ={}&k-\alpha\left[
 \frac{\mathsf Z_{k-1}}{t}
 +\frac{2\sum_{j=0}^{k-2}\mathsf Z_j}{t(t+2)}
 \right]-\frac h2\mathsf Z_{k-1},
\end{align*}
with the two displayed expressions \eqref{eq:Z-k} and
\eqref{eq:sum-Z-k}; expansion of the two squares in
\eqref{eq:cut-residual} reproduces these four terms one for one.
Finally, the cut definition \eqref{eq:q-def} says
\(\alpha q_{k-1}=k-\mathcal C_k\).  Cancelling \(\alpha\) in
\eqref{eq:cut-residual} proves \eqref{eq:q-closed}.

\subsection{The first two horizons}
\label{app:small-horizons}

For \(n=1\), the root is \(r=\sqrt2-1\), \(h=\sqrt2\), and the sole
weight is \(p_0=1/\alpha\).  There are no adjacent weights.  The
certificate matrix and its two required summaries are
\[
 C=\frac1\alpha
 \begin{pmatrix}-r&r\\1&-1\end{pmatrix},
 \qquad
 \one^\top C=(1,-1),
 \qquad
 CB+B^\top C^\top
 =\operatorname{diag}\!\left(\alpha,-(1+r^{-1})\right).
\]

For \(n=2\), the root is \(r=1/2\), \(h=3/2\), and the construction gives
\[
 p_0=q_0=\frac13,\qquad p_1=\frac{13}{3}.
\]
The assembled matrix is
\[
 C=
 \begin{pmatrix}
 -\frac12&\frac13&\frac16\\[1mm]
 \frac76&-4&\frac{17}{6}\\[1mm]
 \frac13&\frac{14}{3}&-5
 \end{pmatrix}.
\]
Every off-diagonal entry is positive, every row sum is zero, and direct
multiplication gives
\[
 \one^\top C=(1,1,-2),
 \qquad
 CB+B^\top C^\top=\operatorname{diag}\!\left(\frac12,\frac12,-10\right).
\]
These are exactly \eqref{eq:b} and \eqref{eq:H-target}.

\subsection{Exact symbolic audit}
\label{app:symbolic-audit}

The supplementary source archive contains
\path{supplement/exact_interval_certificate_checks.py}.  It checks,
over exact symbolic expressions, the interval diagonal and cut formulas,
the \(G\)-to-\(K\) telescoping identity, the suffix closed form including
the \(i=0,1\) boundaries, the positive factor
\eqref{eq:DE-gap}, the Abel sum \eqref{eq:cut-residual}, and the two small
horizons above.  A second script assembles the matrices numerically for a
range of horizons.  These computations are reproducibility aids only; the
proof is the algebra given in the main text and this appendix.

\end{document}